\documentclass[11pt]{amsart}
\usepackage{amssymb,amsmath,amsthm,amsfonts,amsopn,url,color,hyperref,enumitem}
\usepackage[all]{xy}

\theoremstyle{plain}
\newtheorem{thm}{Theorem}[section]
\newtheorem{prop}[thm]{Proposition}
\newtheorem{lemma}[thm]{Lemma}

\newtheorem{cor}[thm]{Corollary}

\theoremstyle{definition}
\newtheorem{defn}[thm]{Definition}
\newtheorem*{defn*}{Definition}
\newtheorem*{question*}{Question}

\newtheorem{example}[thm]{Example}
\newtheorem*{example*}{Example}
\newtheorem{rmk}[thm]{Remark}
\newtheorem*{rmk*}{Remark}
\newcommand{\field}[1]{\mathbb{#1}}

\newcommand{\Z}{\field{Z}}

\newcommand{\ideal}[1]{\mathfrak{#1}}
\newcommand{\m}{\ideal{m}}
\newcommand{\n}{\ideal{n}}
\newcommand{\p}{\ideal{p}}
\newcommand{\q}{\ideal{q}}

\newcommand{\func}[1]{\mathrm{#1} \,}

\newcommand{\Spec}{\func{Spec}}

\newcommand{\hgt}{\func{ht}}

\newcommand{\ra}{\rightarrow}

\newcommand{\li}
 {\leftfootline}

\renewcommand{\phi}{\varphi}
\DeclareMathOperator{\Frac}{Frac}

\DeclareMathOperator{\chr}{char}

\DeclareMathOperator{\Min}{Min}
\newcommand{\Rone}{(R$_1$)}
\newcommand{\Ronei}{$($\emph{R}$_1)$}

\author{Neil Epstein}
\address{Department of Mathematical Sciences \\ George Mason University \\ Fairfax, VA  22030}
\email{nepstei2@gmu.edu}

\author{Jay Shapiro}
\address{Department of Mathematical Sciences \\ George Mason University \\ Fairfax, VA  22030}
\email{jshapiro@gmu.edu}

\title{Perinormality in pullbacks}

\date{\today}
\begin{document}
\begin{abstract}
We further develop the notion of perinormality from \cite{nmeSh-peri}, showing that it is preserved by many pullback constructions.  In doing so, we introduce the concepts of \emph{relative perinormality} and \emph{fragility} for ring extensions.
\end{abstract}

\maketitle

\section{Introduction}\label{sec:intro}

In \cite{nmeSh-peri}, we introduced the notion of \emph{perinormality} for integral domains.  A domain $R$ is \emph{perinormal} if for \emph{local} overrings $L$ of $R$, going-down over $R$ coincides with being a localization of $R$.  Equivalently \cite[Proposition 2.4]{nmeSh-peri}, $R$ is perinormal if for (arbitrary) overrings of $R$, flatness over $R$ and going-down over $R$ coincide.

In this article, we show that perinormality is preserved (or at least retained) in a number of pullback constructions.  We begin in \S\ref{sec:rel} by introducing the concept of \emph{relative} perinormality, and show that it lifts (see Proposition~\ref{pr:relative}).  We use this observation to show that certain rings of the form $R+XS[X]$, $R+XS[\![X]\!]$, and $D+M$ are perinormal (see Corollary~\ref{cor:polyabs} and Theorem~\ref{thm:d+m}) whenever the base ring is.

In \S\ref{sec:fragile}, which comprises the bulk of the paper,  we introduce the concept of what it means for a subring inclusion to be \emph{fragile}.  We show, under mild assumptions, that the pullback of a factor ring of a generalized Krull domain along a fragile ring inclusion is perinormal (see Theorem~\ref{thm:fragilepullback}).  In Theorem~\ref{thm:minprimes}, we show how to contract finite sets of minimal primes in many familiar classes of Noetherian normal domains in a fragile way.  We give a partial converse in Proposition~\ref{pr:subfield}.  In Example~\ref{ex:badpullback}, we show how Theorem~\ref{thm:fragilepullback} may be used while showing that the converse to Proposition~\ref{pr:relative} fails.  This background allows us to demonstrate that there is no canonical ``perinormalization'' of a domain akin to the integral closure operation (see Example~\ref{ex:noperinormalization}).

Finally in \S\ref{sec:hyper}, we show that algebraically contracting a hypersurface to a point typically results in a perinormal (albeit non-Noetherian) domain (see Theorem~\ref{thm:hyper}).

\section{Relative perinormality}\label{sec:rel}
\begin{defn}
Let $A\subseteq B$ be an extension of rings.    We say that $A$ is {\it perinormal in} $B$, if given any prime ideal $\p$ of $A$, $A_\p$ is the only local ring between $A_\p$ and $B_\p$ that is centered on $\p A_\p$ and satisfies going-down over $A_\p$.
\end{defn}

We note some basic facts.   To say an integral domain is perinormal simply means that $A$ is perinormal in frac$(A)$, the field of quotients of $A$.  Any ring is perinormal in itself.  A field is never perinormal in a field extension unless the fields are equal.   If $A$ is perinormal in $B$, then $A_W$ is perinormal in $B_W$ for any multiplicatively closed subset $W$ of $A$.   Let $A$ be a domain and $A\subseteq K$, where $K$ is a field.    If $A$ is perinormal in $K$, then $K$ is the fraction field of $A$ (and so $A$ is perinormal).  To see why, let $W =A \setminus \{0\}$.   Thus $A_W= \Frac(A)$ is perinormal in $K_W = K$.  As noted previously, this implies the two fields are equal.

Let $R$ be a ring and $I$ an ideal of $R$.  Then $V_R(I)$, or simply $V(I)$ if there is no chance for confusion, denotes the set $\{P\in \Spec R: I \subseteq P\}$. We first present a lemma on pullbacks that collects some known results that will be used throughout the paper.  Consider the following pullback diagram
\[\xymatrix{
R\ \ar@{^(->}[r]^f \ar@{->>}[d]^\alpha &S\ar@{->>}[d]^\beta\\
A\ \ar@{^(->}[r]^g &B
}\]
  where $g$ (and hence $f$) is injective, while $\beta$ (and hence $\alpha$) is surjective.
  Viewing $R$ as a subring of $S$, we see that ker $\beta =$ ker $\alpha$ is an ideal of $R$.  In this case, we say that $R$ is the \emph{pullback} of (the diagram) $A \rightarrow B \leftarrow S$.  Note that all of the pullbacks in this paper will satisfy these assumptions.

\begin{lemma} \label{lem:basic}  Given $R$, $S$, $A$ and $B$ as in the above diagram, let $J=$ ker $\alpha$.

\begin{enumerate}
\item\label{it:poiso} There is a poset isomorphism between $(\Spec S)\setminus V_S(J)$ and $(\Spec R)\setminus V_R(J)$ given by the contraction map.

\item\label{it:locoutside} For any $P\in (\Spec S)\setminus V_S(J)$,  the canonical map $R_{P\cap R}\to S_P$ is an isomorphism.

\item\label{it:loc} If $W$ is any multiplicatively closed subset of $R$, then $R_W$ is the pullback of $A_{\overline{W}} \to B_{\overline{W}} \leftarrow S_W$, where $\overline{W}$ denotes the image of $W$ in $A$.
    \end{enumerate}
\end{lemma}
\begin{proof} (1) and (3) are Corollary 1.5 (3) and Proposition 1.9 respectively of  \cite{Fo-topdef}, while (2) is \cite[Lemma 4.4]{nmeSh-peri}.
\end{proof}

\begin{prop}\label{pr:relative} Let $A\subseteq B$ be a ring extension and let $J$ be an ideal of $B$ that is contained in $A$.  If  $A/J$ is perinormal in $B/J$, then $A$ is perinormal in $B$.
\end{prop}

\begin{proof}
Suppose that $A/J$ is perinormal in $B/J$.   Clearly $\Spec A$ is the union of the sets $X$ and $Y$, where $X$ consists of those primes that contain $J$ and $Y$ are the primes that do not contain $J$.  Then the natural map $\Spec(A/J) \to X$ is clearly a bijection, while the contraction map gives a bijection from $(\Spec B)\setminus V_B(J)$ to $Y$ (by Lemma \ref{lem:basic} (\ref{it:poiso})).
Furthermore, if $\p\in Y$, then $A_\p = B_{A\setminus \p}$ (by Lemma \ref{lem:basic} (\ref{it:locoutside})).  Hence we only need  check the extensions $A_\p \subseteq B_{A\setminus\p}$ for $\p \in X$.  Let $T$ be a local ring between $A_\p$ and $B_{A\setminus \p}$ centered on $\p A_\p$, that also satisfies going-down over $A_\p$.  Clearly $T$ contains $JA_\p =JB_{A\setminus \p}$.   Hence we have the chain of rings
$$ (A/J)_\p \subseteq T/JA_\p \subseteq (B/J)_{A\setminus \p} $$
  Then $T/JA_\p$ is a local ring which is  centered on the unique maximal ideal of $(A/J)_\p$ and which satisfies going-down over this ring.  Since $A/J$ is perinormal in $B/J$, we must have $(A/J)_\p =T/JA_\p$.   Since $(A/J)_\p = A_\p/JA_\p$ we conclude that $A_\p = T$, which finishes the proof.
\end{proof}

The above proposition provides many ways to construct perinormal inclusions (and in some cases, new perinormal domains from old).  We provide some cases below.

\begin{cor}\label{cor:poly}
Let $R \subseteq S$ be a perinormal ring inclusion.  Let $X$ be an analytic indeterminate.  Then $R+XS[X]$ is perinormal in $S[X]$, and $R+XS[\![X]\!]$ is perinormal in $S[\![X]\!]$.
\end{cor}

Next, we will use the above result to prove that certain classical pullback construction yields perinormal domains under reasonable assumptions.  First, we recall said reasonable assumptions:

\begin{thm}\label{thm:equiv}\cite[Theorem 4.7]{nmeSh-peri}
Let $R$ be a universally catenary Noetherian domain with fraction field $K$ and integral closure $S$.  The following are equivalent \begin{enumerate}[label=\emph{(\alph*)}]
\item $R$ is perinormal.
\item For each $\p \in \Spec R$, $R_\p$ is the only ring $S$ between $R_\p$ and $K$ such that the induced map $\Spec S \rightarrow \Spec R_\p$ is an order-isomorphism.
\item $R$ satisfies \Ronei, and for each $\p \in \Spec R$, $R_\p$ is the only ring $S$ between $R_\p$ and $S_\p$ such that the induced map $\Spec S \rightarrow \Spec R_\p$ is an order-isomorphism.
\end{enumerate}
\end{thm}

Recall that by definition, a commutative ring $R$ satisfies \Rone\ if for every height one prime $\p$ of $R$, $R_\p$ is a (not necessarily discrete) valuation domain.

\begin{cor}\label{cor:polyabs}
Let $R$ be a Noetherian, universally catenary, 
perinormal integral domain, whose integral closure $S$ is finitely generated as an $R$-module.  Let $X$ be an analytic indeterminate.  Then $R+XS[X]$ and $R+XS[\![X]\!]$ are perinormal.
\end{cor}

\begin{proof}
First we show that the rings in question are themselves universally catenary.  Let $s_1=1, s_2, \ldots, s_n$ be a generating set for $S$ as an $R$-module.  For $A=R+XS[X]$, note that $S[X]$ is generated as an $A$-module by the same set of generators that generates $S$ as an $R$-module.  Then since $S[X]$ is finitely generated over $R$, the Artin-Tate lemma \cite{ArTa-lemma} shows that $A$ must be finitely generated over $R$, and hence universally catenary and Noetherian.  Next, let $B=R+XS[\![X]\!]$.  Observe that $B \cong R[\![T_1, \ldots, T_n]\!] / (T_1-s_1X, \ldots, T_n - s_n X)$.   By \cite[Th\'eor\`eme 1.12]{Sey-chain}, $R[\![T_1, \ldots, T_n]\!]$ is universally catenary; hence its quotient ring $B$ must be as well.

Let $S' = S[X]$ (resp. $S[\![X]\!]$) and $T := R+P$, where $P=XS'$.  We claim that $T$ is \Rone.  By Lemma~\ref{lem:basic}(\ref{it:locoutside}), we need only check the height one primes of $T$ that contain $P$.  But $P$ is itself a height one prime, so we need only check that $T_P$ is a valuation ring, which follows since $T_P = S'_P$ and $P$ is a height one prime of the Krull domain $S'$, proving the claim.


Since $T$ is perinormal in its integral closure (by Corollary~\ref{cor:poly}) and \Rone, it satisfies condition (c) of Theorem~\ref{thm:equiv}.  But then since $T$ is universally catenary, said theorem implies that $T$ is perinormal.
\end{proof}

\begin{thm}[The classical $D+M$ construction, and more]\label{thm:d+m}
Let $D$ be an integral domain, $k=\Frac D$ its fraction field, and $V$ a valuation domain with residue field $k$.  Then the pullback $R$ of the diagram $D \rightarrow k \leftarrow V$ is perinormal \emph{if and only if} $D$ is perinormal.
\end{thm}

\begin{proof}
Let $\m = \ker (V \rightarrow k)$, the conductor of the pullback square.

First suppose $D$ is not perinormal.  Let $(T,\n)$ be a local overring of $D$ such that $D \subseteq T$ is going-down and $T$ is not a localization of $D$.  Let $S$ be the pullback of the diagram $T \rightarrow k \leftarrow V$, thought of as a subring of $V$ that contains $R$.  Then $S$ is an overring of $R$ (since $V$ is) which is local with unique maximal ideal the preimage of $\n$.  To see that the extension $R \subseteq S$ satisfies going-down, let $P_1 \subseteq P_2$ be primes of $R$ and $Q_2$ a prime of $S$ with $Q_2 \cap R = P_2$.  If $P_1 \nsupseteq \m$, then by Lemma \ref{lem:basic}(\ref{it:poiso}), $P_1=Q_1\cap R$ for some $Q_1 \in \Spec V \setminus \{\m\}$.   As $V$ is local, we have $Q_1\subset \m$, so in fact $Q_1 =P_1$ is an ideal of $R$ and hence of $S$.     In particular, $P_1$ is a prime of $S$, solving the going-down problem.  On the other hand, if $P_1 \supseteq \m$, let $\p_j := P_j / \m \in \Spec D$ for $j=1,2$ and let $\q_2 := Q_2 /\m \in \Spec T$.  Since $D \subseteq T$ satisfies going-down, there is some $\q_1 \in \Spec T$ with $\q_1 \subseteq \q_2$ and $\q_1 \cap D =\p_1$.  But setting $Q_1$ to be the preimage of $\q_1$ in $S$, we have $Q_1\subseteq Q_2$ and $Q_1 \cap R = P_1$, completing the proof that  $S$ satisfies going-down over $R$. Hence to show that $R$ is not perinormal, it suffices to show that $S$ is not a localization of $R$.    Suppose otherwise, say $S=R_W$ for some multiplicatively closed set $W\subset R$.  Since $\m$ is a proper ideal of both $R$ and $S$, it follows that $\m\cap W=\emptyset$.  Thus if $W'$ is the image of $W$ in $D=R/\m$, we would have $T=D_{W'}$, contradicting our choice of $T$.  Thus $R$ is not perinormal.

Conversely, suppose $D$ is perinormal.  Let $(T,\n)$ be a going-down local overring of $R$.  If $\n \cap R \nsupseteq \m$, then by Lemma \ref{lem:basic} (\ref{it:poiso}) and (\ref{it:locoutside}), $R_{\n \cap R}$ is in fact a localization of $V$, and hence also a valuation ring.  Thus the overring $T$ must be a localization of $R_{\n \cap R}$ and therefore of $R$.  So we may assume that $\m \subseteq \n$.  In this case, we claim that $T \subseteq V$.  If not, then let $a\in T \setminus V$.  Then since $V$ is a valuation ring, $b=a^{-1} \in V$.  But since $a=b^{-1} \notin V$, $b$ is a nonunit of $V$, whence $b \in \m \subseteq \n \subseteq T$.  Since also $a=b^{-1}\in T$, we have that $b$ is a unit of $T$ contained in its maximal ideal $\n$, which is absurd.  Hence, $R \subseteq T \subseteq V$ and $T$ is local and going-down over $R$.  Then since $R$ is perinormal in $V$ (by Proposition~\ref{pr:relative}, since $D$ is perinormal in $k$), it follows that $T$ is a localization of $R$.
\end{proof}

\begin{rmk}
The above subsumes the classical $D+M$ construction (see for example \cite[Appendix 2]{Gil-old}).  There, one starts with a valuation domain $V$ that contains a field $k$ as an algebra retract, so that $V=k+M$; then one takes a subring $D \subseteq k$ and sets $R=D+M$.   Recall that $D+M$ is never Noetherian unless $D$ is a field and $k$ is finite dimensional over $D$. (cf. \cite[Theorem A(m), p. 562]{Gil-old}).

However, the above theorem covers additional cases.  Recall the theorem of Hasse and Schmidt (cf. \cite[Theorem 29.1]{Mats}) that says that for any prime characteristic field $k$, there is a characteristic zero rank one discrete valuation ring $V$ (which necessarily contains no field at all) whose residue field is $k$.  So using the two theorems together, one may start with one's favorite prime characteristic domain $D$, let $k$ be its fraction field, use Hasse-Schmidt to find a characteristic zero DVR $V$ whose residue field is $k$, and form the pullback $R$.  Then $R$ is perinormal if and only if $D$ is.
\end{rmk}

\section{Pullbacks of fragile subrings}\label{sec:fragile}

Later, after we have one of our main results on pullbacks, we will construct an example to show that the converse to Proposition~\ref{pr:relative} is false (see Example~\ref{ex:badpullback}).   First we develop some basic results as prelude to our next theorem.

\begin{lemma} \label{lem:goingdown} Let $f: R \rightarrow S$ and $g: S\rightarrow T$ be ring maps and let $J$ be an ideal of $S$.    Suppose that the induced map $\Spec S \rightarrow \Spec R$ is injective  on $\Spec S\setminus V(J)$ and that the image of the induced map $\Spec T \rightarrow \Spec S$ is contained in $\Spec S\setminus V(J)$.  If $g\circ f: R \rightarrow T$ satisfies going-down, so does $g:S \rightarrow T$.
\end{lemma}
\begin{proof} Let $P_1\subset P_2$ be prime ideals of $S$ and let $Q_2$ be a prime ideal of $T$ with $g^{-1}(Q_2) =P_2$.  Let $\p_i=f^{-1}(P_i)$ for $i=1,2$.    By the going-down hypothesis on the map $g \circ f$, there is a prime ideal $Q_1$ of $T$ with $Q_1 \subseteq Q_2$ and $(g\circ f)^{-1}(Q_1) = \p_1$.   Clearly $g^{-1}(Q_1), P_1, P_2\in \Spec S\setminus V(J)$.  But then we have $f^{-1}(P_1) = \p_1 = (g \circ f)^{-1}(Q_1) = f^{-1}(g^{-1}(Q_1))$, so by injectivity of the map in question, we have $P_1 = g^{-1}(Q_1)$, completing the proof.
\end{proof}

We introduce the next definitions so as to simplify the statement of the main theorem of the current section.  They  may also be of independent interest.

\begin{defn}\label{def:fragile}
Let $A \subseteq B$ be an integral extension of commutative rings.  We say that the extension is \begin{itemize}
\item \emph{apparently fragile} (or \emph{a.f.}) if for any ring $C$ with $A \subsetneq C \subseteq B$, there exists some minimal prime $\p$ of $A$ and primes $P,Q \in \Spec C$ with $P \neq Q$ and $P \cap A = Q\cap A = \p$.
\item \emph{fragile} if $A_P \subseteq B_{A \setminus P}$ is apparently fragile for any $P \in \Spec A$.
\item \emph{globally fragile} if $A_W \subseteq B_W$ is apparently fragile for any multiplicative subset $W$ of $A$.
\end{itemize}
\end{defn}

Note that in all the above definitions, one allows the vacuous case $A=B$.


Next, recall the following definition from \cite[\S 35]{Gil-old}, where for a commutative ring $R$, we set $\Spec^1(R) := \{\p \in \Spec R \mid \hgt \p=1\}$:
\begin{defn}
An integral domain $R$ is a \emph{generalized Krull domain} if 
\begin{enumerate}
\item $R = \bigcap_{\p \in \Spec^1(R)} R_\p$,
\item For any nonzero element $r\in R$, the set $\{\p \in \Spec^1(R) \mid r\in \p\}$ is finite, and
\item $R$ satisfies \Rone.
\end{enumerate}
%
\end{defn}


\begin{lemma} \label{lem:relperi} Let $S$ be a generalized Krull domain and let $J$ be an ideal of $S$ where $J$ has finitely many minimal primes, all of height at least two.   Let $A$ be a subring of $S/J$ over which $S/J$ is integral.  Let $R$ be the pullback of $A\to S/J \leftarrow S$.  Then
  \begin{enumerate}
 \item The Spec map induces a bijection $ \Spec^1S\to \Spec^1R$ where the corresponding localizations coincide. In particular $R$ satisfies \Ronei.
 \item If the extension $A \subseteq S/J$ is apparently fragile, then no ring  between $R$ and $S$ satisfies going-down over $R$, other than $R$ itself.
 \end{enumerate}
\end{lemma}
\begin{proof} (1)
By Lemma \ref{lem:basic}(\ref{it:poiso}) the contraction map $\Spec S \to \Spec R$ gives an order preserving bijection between the prime ideals of $S$ that do not contain $J$ and the prime ideals of $R$ that do not contain $J$.  Since no height one prime of $S$ contains $J$, every height one prime of $S$ contracts to a height  one prime of $R$.
  Moreover  by integrality each height one prime of $R$ is lain over by a prime of $S$ (necessarily of height one by INC).   Thus the Spec map induces a bijection $ \Spec^1S\to \Spec^1R$ where, by Lemma \ref{lem:basic} (\ref{it:locoutside}), the corresponding localizations coincide.  The ``in particular" statement now follows since $S$ is a generalized Krull domain.

(2) Let $T$ be a ring with $R \subsetneq T \subseteq S$ that  satisfies going-down over $R$.   Then $J$ is an ideal of $T$ and $S$ is integral over $T$.  By assumption, there is some minimal prime $\p$ of $J$ in $R$ such that there exist distinct primes $P_1, P_2 \in \Spec T$ with $P_1 \cap R = P_2 \cap R = \p$, $i=1, 2$.  By integrality there exists a prime ideal $Q_1$ of $S$ that  contracts to $P_1$, which by INC must be a minimal prime over $J$.   Similarly, note that each  prime ideal of $S$ that contracts to $P_2$  must be minimal over $J$, whence the set of such primes is finite (by our assumption on $J$).  Say $Q_2,...,Q_r$ is a complete list of all primes of $S$ contracting to $P_2$.   We first claim that there exists a height one prime ideal of $S$ that is contained in $Q_1$, but not in any of the $Q_i$, $i=2,3,...,r$.  To see this note that both $S_{Q_1}$ and $S_W$, where $W=S\setminus \bigcup_{j=2}^r Q_j$, are generalized Krull domains.  Let $X$ be the set of height one primes of $S$ contained in $Q_1$ and let $Y$ be the set of height one primes of $S$ contained in $\bigcup_{j=2}^r Q_j$.  Hence $S_{Q_1} =\bigcap_{Q\in X} S_Q$ and $S_W= \bigcap_{P\in Y} S_P$.   Therefore, if the claim is false, i.e., if  $X\subseteq Y$, it follows that $S_W \subseteq S_{Q_1}$.   This in turn implies that $Q_1 \subset \bigcup_{j=2}^r Q_j$.   By the prime avoidance theorem, we can conclude that $Q_1\subseteq Q_j$ for some $j\geq 2$.    But this
contradicts our choice of the $Q_i$'s, whence the claim is proved.

Let $Q\in X\setminus Y$.  By (1), $Q$ is the unique prime ideal of $S$ that contracts to  $\q = Q\cap R\subset \p$.     We next claim that $Q\cap T$ is not contained in $P_2$.  If it were, then by going-up $Q$ is contained in a prime ideal of $S$ that contracts to $P_2$, a contradiction to our choice of $Q$. As $J$ is an ideal of $T$,  it follows from part (1) that the Spec map induces a bijection $\Spec^1S\to \Spec^1T$, which means that we also have a bijection $\Spec^1T\to \Spec^1R$.  Thus there is no prime ideal of $T$ that is contained in $P_2$ which contracts to $\q$.  But this 
contradicts the going-down property of $R\subset T$, since $P_2\cap R = \p$.  Thus $T$ cannot exist.
\end{proof}


We are ready for the main theorem of the current section.

\begin{thm} [Pullbacks of fragile subrings] \label{thm:fragilepullback} Let $S$ be a generalized Krull domain and let $J$ be an ideal of $S$ where $J$ has $n$  minimal primes with $n\geq 2$.   Assume further that $J$ has height at least two.   Let $A$ be an integral fragile subring of $S/J$.     If $R$ is the pullback $A\to S/J \leftarrow S$, then $R$ is perinormal.  
\end{thm}
\begin{proof} First observe that $S$ is integral over $R$.  To see this, let $s\in S\setminus R$ and let $t\in S/J$ be its homomorphic image.  Then there is a monic polynomial $F(X) \in A[X]$ such that $F(t)=0$.  Lift $F$ (via its coefficients) to a monic polynomial $f(X) \in R[X]$.  Then $f(s) =j$ for some $j\in J$.  Set $g(X) := f(X) - j$; this is a monic polynomial over $R$ that vanishes when evaluated at $s$.

Now let $(T,\n)$ be a local overring of $R$ that satisfies going-down over $R$
.  We will show that $T$ is a localization of $R$. \vspace{3pt}

\noindent {\bf Case 1:}  Suppose $JT = T$.  Then we claim $S\subseteq T$.  To see this,  
note that since $J \nsubseteq \n$, we have $\n \cap R = \p \nsupseteq J$.  Then there is some $P \in \Spec S$ with $P \cap R = \p$ (since $S$ is integral over $R$), and $P \nsupseteq J$ either, whence  we have $R_\p = S_P$ by Lemma~\ref{lem:basic} (\ref{it:locoutside}).  Hence, $S \subseteq S_P = R_\p \subseteq 
T$.

Also since $JT=T$, we have the image of the map $\Spec T\to \Spec S$ is contained in $\Spec S\setminus V_S(J)$.  By Lemma~\ref{lem:basic} (\ref{it:poiso}), the contraction map $\Spec S\to \Spec R$ is injective on the set $\Spec S\setminus V(J)$.   Hence, by Lemma~\ref{lem:goingdown}, $S\subseteq T$ satisfies going-down.  Therefore, since $S$ is 
perinormal, $T
=S_P$ for some $P \in \Spec S$.  Then $JS_P = S_P$ by assumption, so that $J
\nsubseteq P$.  But then since $J$ is a shared ideal of $S$ and $R$, Lemma~\ref{lem:basic} (\ref{it:locoutside}) applies, showing that 
$S_P = R_{P\cap R}$, so that $T$ is a localization of $R$. \vspace{3pt}

\noindent {\bf Case 2:}  On the other hand if $JT \neq T$, then by \cite[Proposition 3.9]{nmeSh-peri}, the map $\Spec T \rightarrow \Spec R_{\n\cap R}$ induces a bijection $\Spec^1 T \to \Spec^1(R_{\n \cap R})$, and by \cite[Lemma 3.7]{nmeSh-peri} the corresponding localizations are equal.  Since we have a similar bijection between  $\Spec^1(S_W)$ and $\Spec^1(R_{\n\cap R})$ (where $W = R\setminus \n$) by Lemma~\ref{lem:relperi} (1), we get
 \[R_{\n\cap R} \subseteq T \subseteq \bigcap_{P \in \Spec^1T} T_P = \bigcap_{Q \in \Spec^1(S_W)} S_Q = S_W
\]
where the last equality follows from $S$ being a generalized Krull domain.

Now let $\p$ be the homomorphic image of $\n \cap R$ in $A$.  By Lemma~\ref{lem:basic} (\ref{it:loc}), $R_{\n \cap R}$ is the pullback of the diagram $A_\p \ra (S/J)_\p \leftarrow S_W$, and by hypothesis, the extension $A_\p \subseteq (S/J)_\p$ is integral and apparently fragile.  
But then since $S_W$ is the integral closure of $R_{\n\cap R}$, it follows from Lemma~\ref{lem:relperi} (2) that $T=R_{\n\cap R}$.
\end{proof}

At this point, the reader may wonder under what circumstances fragile extensions exist.  The reader may also ask for conditions under which fragility, global fragility, and apparent fragility coincide.  The next few results and examples are dedicated to these issues.  First, note that in exploring any kind of fragility, one need only ask what happens with reduced rings.

\begin{lemma}[Reducing fragilities to the reduced case]\label{lem:redred}
Let $A \subseteq B$ be an integra  l extension.  The extension is fragile (resp. apparently fragile, resp. globally fragile) if and only if $A$ and $B$ have the same nilradical $J$ and the extension $A/J \hookrightarrow B/J$ is fragile (resp. apparently fragile, resp. globally fragile).
\end{lemma}

\begin{proof}
First suppose $A \subseteq B$ is apparently fragile, and let $J$ be the nilradical of $B$.  Note that the nilradical of $A$ is $J \cap A$.  From this, it is easy to see that the map $\Spec(A+J) \ra \Spec A$ is a poset isomorphism, so by apparent fragility, we have $J \subseteq A$, so that $J \cap A =J$.  Now let $C$ be a ring with $A/J \subsetneq C \subseteq B/J$.  Then $C=D/J$ for some ring $D$ with $A \subsetneq D \subseteq B$, so there is some $\p \in \Min A$ and unequal $P, Q \in \Spec D$ with $P \cap A = Q \cap A = \p$.  Quotienting out by $J$ then shows that $C$ has the same relationship with $A/J$, completing the proof that $A/J$ is apparently fragile in $B/J$.

Now assume $A \subseteq B$ is globally fragile (resp. fragile).  Note that every prime ideal of $A$ contains $J \cap A$, and every multiplicative set $W$ in $A$ for which $A_W \neq 0$ satisfies $W \cap A = \emptyset$.  Hence, for any $W$ (resp. for any $W=A\setminus P$, $P \in \Spec A$), the fact that $A_W \subseteq B_W$ is apparently fragile implies (by the above) that $JB_W \subseteq A_W$ for any such $W$.  Hence $J \cap B \subseteq A$ is the nilradical of $A$.  Moreover, by the above, we have that $A_W/J_W = (A/J)_W \hookrightarrow (B/J)_W = B_W/J_W$ is apparently fragile for any such $W$.  Hence $A/J \hookrightarrow B/J$ is globally fragile (resp. fragile).

Conversely, suppose that $J\cap B \subseteq A$ and that $A/J \hookrightarrow B/J$ is apparently fragile. Let $C$ be a ring with $A \subsetneq C \subseteq B$.  Then $J$ is an ideal of $C$ and $A/J \subsetneq C/J \subseteq B/J$.  Some minimal prime of $A/J$ splits in $C/J$; hence the corresponding minimal prime of $A$ splits in $C$, so that $A \subseteq B$ is apparently fragile.

Finally, suppose that $J \cap B \subseteq A$ and that $A/J \hookrightarrow B/J$ is globally fragile (resp. fragile).  Then for any multiplicative set $W$ in $A$ with $W \cap J = \emptyset$ (resp. $W=A\setminus P$ for some $P \in \Spec A$), $J_W$ is the nilradical of both $A_W$ and $B_W$, and $(A/J)_W = A_W/J_W \hookrightarrow B_W / J_W= (B/J)_W$ is apparently fragile.  Then by the above, $A_W \subseteq B_W$ is apparently fragile for any such $W$, so that $A \subseteq B$ is globally fragile (resp. fragile).
\end{proof}

Next, we observe that the three notions of fragility coincide in many dimension 0 cases.

\begin{prop}[Fragilities coincide in dimension 0]\label{pr:dim0}
Let $A$ be a dimension 0 ring with finitely many minimal primes, and let $B$ be an integral extension of $A$.  Then $A\subseteq B$ is fragile iff it is apparently fragile iff it is globally fragile.
\end{prop}

\begin{proof}
First, we may assume by Lemma~\ref{lem:redred} that $A$ and $B$ are reduced, so that $A$ is a product of finitely many fields.  Next, note that is enough to show that apparent fragility implies global fragility.  So suppose the extension is apparently fragile.  Setting some notation, we have $A = K_1 \times \cdots \times K_n$, where the $K_j$ are fields, and the prime ideals of $A$ are $\p_j = K_1 \times \cdots \times K_{j-1} \times 0 \times K_{j+1} \times \cdots \times K_n$, for $1\leq j \leq n$.  By the Chinese Remainder Theorem, we have $B \cong \prod_{i=1}^n B/\p_i B$.

Now let $W$ be a saturated multiplicative subset of $A$.  Without loss of generality, there is some $1\leq t \leq n$ such that $W = A \setminus (\p_1 \cup \cdots \cup \p_t)$.  We have $A_W = A/(\p_1 \cap \cdots \cap \p_t) = K_1 \times \cdots \times K_t$ and $B_W \cong \prod_{i=1}^t B/\p_i B$ under the  earlier identification of $B$ as a product of $K_i$-vector spaces.  Let $D$ be a ring with $A_W \subsetneq D \subseteq B_W$.  Then $D=C_W$ for some ring $C$ with $A \subsetneq C \subseteq B$.  Moreover, any element of $C_W$ has a numerator in $C$ that can be represented in the form $(b_1, \cdots, b_t, a_{t+1}, \cdots, a_n)$, with the $b_i \in B/\p_i B$ and the $a_i \in K_i$.  In particular, for any $i>t$, there is only one prime ideal of $C$ that contracts to $\p_i$.  So by apparent fragility, there is some $i\leq t$ and some $P, Q \in \Spec C$ with $P \neq Q$ and $P \cap A = Q \cap A = \p_i$.  But then $\p_i A_W$ is a (minimal) prime of $A_W$, and we have $PC_W, QC_W \in \Spec C_W$ with $PC_W \neq QC_W$ and $PC_W \cap A_W = QC_W \cap A_W = \p_i A_W$.  Since $W$ was arbitrary, $B$ is a globally fragile extension of $A$.
\end{proof}

Next, we show that the three notions of fragility coincide when the base ring is an integral domain, and more generally.

\begin{prop}\label{pr:contract1}
Let $A \subseteq B$ be an apparently fragile, integral extension.  Assume that there is some $\p \in \Min A$ such that \begin{enumerate}
\item only finitely many primes in $B$ lie over $\p$, and
\item for any $\q \in \Min A \setminus \{\p\}$, there is only one prime ideal of $B$ that contracts to $\q$.
\end{enumerate}
Then $\p B=\p$, and the extension $A\subseteq B$ is globally fragile.
\end{prop}

\begin{proof}
Let $P_1, \ldots, P_n$ be the primes of $B$ (necessarily minimal) that contract to $\p$.  Set $J := \bigcap_{i=1}^n P_j$.  Consider the ring $C := A + J$.  It fits into the following pullback diagram: \[\xymatrix{
C\ \ar@{^(->}[r] \ar@{->>}[d] &B\ar@{->>}[d]\\
A/\p\ \ar@{^(->}[r] &B/J
}\]
By Lemma~\ref{lem:basic}(\ref{it:poiso}), the contraction map $\Spec B \ra \Spec C$ induces a poset isomorphism between $\Spec B \setminus V_B(J)$ and $\Spec C \setminus V_C(J)$.  However, 
$\Spec B \setminus V_B(J)$ consists of those primes of $B$ that do not contain any of $P_1, \ldots, P_n$, so that by hypothesis, the \emph{minimal} elements of $\Spec B \setminus V_B(J)$ are in bijection with the minimal elements of $\Spec A \setminus \{\p\}$ via the contraction map $\Spec B \ra \Spec A$.  Moreover, $V_C(J) \cong V_A(\p)$ as posets.  Put together, this means that the contraction map $\Spec C \ra \Spec A$ induces a bijection between the minimal primes of $C$ and those of $A$, so that by apparent fragility of the extension $A\subseteq B$, we have $C=A$.  Thus, $J\subseteq A$, but $\p \subseteq J \cap A \subseteq P_1 \cap A = \p$, so $J=\p$, which since it is also an ideal of $B$ satisfies $\p B = \p$.

Now let $W$ be a multiplicative subset of $A$.  If $W \cap \p \neq \emptyset$, then $A_W = B_W$, since $\p \subseteq (A:B)$.  Assuming then that $W\cap \p = \emptyset$, let $R$ be a ring with $A_W \subsetneq R \subseteq B_W$.  Then $R=C_W$ for some $A \subsetneq C \subseteq B$.  Then by apparent fragility of $A\subseteq B$, without loss of generality we have $Q_1 \neq Q_2$, where $Q_i := P_i \cap C$ for $i=1, \ldots, n$.  Since $W\cap \p =\emptyset$, we have $W\cap Q_j =\emptyset$ as well, so that each $(Q_j)_W \in \Spec R$ and $(Q_1)_W \neq (Q_2)_W$.  Moreover, $(Q_1)_W \cap A_W = (Q_2)_W \cap A_W = \p_W$.  Thus, the extension $A\subseteq B$ is globally fragile.
\end{proof}

\begin{example}\label{ex:CS} Theorem~\ref{thm:fragilepullback} is sharp, in that the height condition can not be omitted.  To see this, consider the classic example given by Cohen and Seidenberg \cite[Section 3(A)]{CohSei}, which they used to show that the integral closure of a domain may not satisfy going down over it.

In particular, the authors there consider the ring $R=k[X,Y,Z]/(Y^2 - X^2 - X^3)$ over a field $k$ of characteristic $0$ (familiar in geometry as the coordinate ring of a cylinder over a nodal curve), along with its integral closure $S = k[t,Z]$.  The map from $R$ to $S$ is given by $x \mapsto t^2-1$, $y \mapsto t^3-t$.  (Here we are using the convention that lower case variables represent the residue classes of upper case variables.)

Put in the context of Theorem~\ref{thm:fragilepullback}, let $J := (t^2 - 1)S = (x,y)R$. Note that as an ideal of $S$, $J$ has two minimal primes, namely $(t-1)S$ and $(t+1)S$.  Then the map $R \hookrightarrow S$ fits into the pullback diagram \[\xymatrix{
R \ar[r] \ar@{->>}[d] &S\ar@{->>}[d]\\
A=k[Z] \ar[r]^g & B=S/J
}\]
where $g$ is the obvious inclusion.  That is, $R = k[Z] + J$.  (In the notation of \cite{DFF-amalgam},  $R = A \bowtie^f J$, where $f:A \ra S$ is the inclusion map $k[Z] \hookrightarrow k[t,Z]$.) Moreover, we claim that $g$ is a fragile inclusion.  To see this, first note that by Proposition~\ref{pr:contract1}, we only need to show it is \emph{apparently} fragile.  Next, observe that as an $A$-module, $B$ is free of rank two on the elements $1$ and $t$; in particular, $B$ is integral over $A$.  Accordingly, choose a ring $C$ with $A \subsetneq C \subseteq B$.  Then $C$ contains an element of the form $e=t\cdot h$, where $h\in A$.  Then $(e+h)(e-h) = e^2 - h^2 = (t^2-1)h^2 = 0$, but $e+h\neq 0$ and $e-h \neq 0$ are elements of $C$.  Hence, $C$ is not an integral domain.  But it is reduced because $B$ is.  Thus, $C$ has at least two minimal primes, even though $A$ contains only one.  Moreover, both must contract to $0$, since both of the minimal primes of $B$ contract to $0$.  This completes the proof that $A$ is apparently (hence globally) fragile in $B$.  Thus, all conditions of Theorem~\ref{thm:fragilepullback} are satisfied except for the height condition.

However, $R$ is not perinormal, as it is not \Rone\ (see \cite[Proposition 3.2]{nmeSh-peri}).  Indeed, $J$ is a height one prime of $R$, and $R_J \cong k(Z)[X,Y]_{(X,Y)} / (Y^2 - X^2 - X^3)$, which is not integrally closed because $(y/x)^2 = 1+x\in R_J$, even though $y/x \notin R_J$.
%
\end{example}

Next we exhibit two broad classes of rings where such fragile subrings exist.

\begin{thm}[How to squeeze minimal primes]\label{thm:minprimes}
Suppose that the commutative ring $B$ satisfies one of the following two conditions: \begin{enumerate}[label=\emph{(\alph*)}]
 \item $B$ is finitely generated over a field $k$, or
 \item $B$ is a complete Noetherian local ring such that either  \begin{itemize}
   \item $B$ is equicharacteristic (i.e. contains a field), or
   \item the residual characteristic of $B$ generates a height 1 ideal.
 \end{itemize}
\end{enumerate}
Let $P_1, \ldots, P_n$ be the minimal primes of $B$, and choose an integer $1\leq t < n$.  Then there is a subring $A$ of $B$ such that $B$ is integral over $A$, the extension $A \subseteq B$ is globally fragile, all the $P_j$ for $j< t$ contract to \emph{distinct} minimal primes of $A$, and $P_i \cap A = P_j \cap A$ for all $t\leq i <j \leq n$.  Moreover, $A$ satisfies the same condition among \emph{(a)}, \emph{(b)} that $B$ satisfies.
\end{thm}

\begin{proof}
Let $J := \bigcap_{i=t}^n P_i$.
Then there is some subring $C$ of $B/J$ where $C$ is a (regular) domain and  $C \subseteq B/J$ is a module-finite extension.  (In case (a), choose a Noether normalization of $B/J$ with respect to $k$, cf. \cite[Lemma 33.2]{Mats}, which will then be a polynomial ring over $k$.  In case (b), one may use \cite[Theorem 16]{Coh-complete}, in which case $C$ will be a complete regular local ring.)  Now let $D$ be the pullback of the diagram $C \hookrightarrow B/J \twoheadleftarrow B$.  Then $B$ is module-finite over $D$ (lift the generators of $B/J$ over $C$ to $B$, and add the generator $1$ if necessary).  We claim that the extension $D \subseteq B$ satisfies the properties desired for $A \subseteq B$ with regard to behavior on minimal primes.

To see the claim, recall the order preserving  bijection between $(\Spec B)\setminus V_B(J)$ and $(\Spec D)\setminus V_D(J)$ given   by the contraction map (Lemma \ref{lem:basic} (\ref{it:poiso})).  Thus for $i=1,\ldots t-1$, $\p_i:=P_i\cap D$ are distinct minimal prime ideals of $A$.  Clearly they are also the only minimal primes of $D$, with the possible exception of $J$.   We next show that $J$ is a minimal prime of $D$.  Since $D/J \cong C$ we see that $J$ is a prime ideal of $D$.  Suppose that it is not minimal, then there exists $\q\subset J$, $\q$ a minimal prime ideal of $D$.    Since $\q\not\in V_D(J)$, we must have $\q = \p_i$ for some $1\leq i<t$.  Since $D\subset B$ satisfies going-up, there must be a prime ideal $Q$ of $B$ with $P_i \subset Q$, such that $Q\cap D = J$.   But then $\bigcap_{i=t}^n P_i = J\subseteq Q$, whence $P_k\subseteq Q$ for some $t\leq k\leq n$.  Which gives the contradiction that $P_i\subseteq P_k$.  Hence $J$ is a minimal prime of $D$ distinct from all the  $\p_i$, for $i=1,2,\ldots, t-1$.

To construct $A$, we note that as the Noetherian ring $B$ is module finite over $D$, $D$ is also a Noetherian ring by the Eakin-Nagata theorem \cite[Theorem 3.7(i)]{Mats}.   Hence we have ACC on chains of rings between $D$ and $B$.   In particular we can pick one that is maximal with respect to 
having its minimal primes in bijective correspondence with those of $D$, under the $\Spec$ map.  Call it $A$.  Then the extension $A \subseteq B$ is apparently fragile, hence globally fragile by Proposition~\ref{pr:contract1}.

If $B$ contains a field, then $A$ contains a field too.  To see this, first note that the characteristics of $A$, $B$, and the field that $B$ contains must match.  If that characteristic is a positive prime number $p$, then both $A$ and $B$ contain $\mathbb F_p$.  If $\chr A = \chr B=0$, then $A$ and $B$ contain the rational integers, and every nonzero integer is invertible in $B$.  But by integrality, every element of $A$ that is invertible in $B$ must also be invertible in $A$.  Thus, $\mathbb Q\subseteq A$.

If $B$ is finitely generated over a field, then by the Artin-Tate lemma \cite{ArTa-lemma}, $A$ is also.   If $B$ has mixed characteristic $p$ and $p$ generates a height one ideal of $B$, then integrality forces $pA$ to be a height one ideal of $A$ as well.

Finally, we must show that if $B$ is complete, Noetherian and local, then $A$ is complete, Noetherian and local.  First, $A$ is Noetherian by the Eakin-Nagata theorem \cite[Theorem 3.7(i)]{Mats}.  Let $\n$ be the maximal ideal of $B$.  By Going-Up, $\m:=\n \cap A$ is maximal.  By the lying-over and incomparability properties of integral extensions, $(A,\m)$ is local.  Next, since $\m B$ is $\n$-primary, the $\n$-adic topology on $B$ is the same as the $\m$-adic one.  Hence, one has the diagram \[
A \hookrightarrow \hat A \hookrightarrow \hat{A} \otimes_A B \cong \hat B = B.
\]
In particular, since $B$ is module-finite over $A$, $\hat A$ must be module-finite over $A$ as well.  Since $A/\m = \hat{A} / \hat \m = \hat{A} / \m \hat{A}$, we have $\hat A = A + \m \hat{A}$, so that by Nakayama's lemma, $\hat A = A$.
\end{proof}

In our next result, by an application of Zorn's Lemma, we are able to show that even in non-Noetherian context, (apparently) fragile subrings can still exist.

\begin{prop}[Expanding subrings to the point of apparent fragility]\label{pr:fragile} Let  $C \subseteq B$ be an integral extension of commutative rings. Then  there exists an apparently fragile subring $A$ of $B$ containing $C$ such that in the map $\Spec A \ra \Spec C$, every minimal element of $\Spec C$ has a singleton fiber.
\end{prop}
\begin{proof} We use Zorn's Lemma to show the existence of $A$.   Let
\[
C\subseteq R_1\subseteq R_2\subseteq \ldots
\]
 be a chain of rings in $B$, chosen in such a way that the maps $\Spec R_j \ra \Spec C$, the elements of $\Min C$ have singleton fibers.   Set $T=\bigcup R_i$.  It will suffice to show that the elements of $\Min C$ have singleton fibers in the map $\Spec T \ra \Spec C$.  So let $\p \in \Min C$ and choose $P, Q \in \Spec T$ with $P \cap C = Q \cap C = \p$.  If $P \neq Q$, then there exists $a\in P \setminus Q$ and $b \in Q \setminus P$.  There exists $R_j$ such that $a, b \in R_j$.  Then setting $P' := P \cap R_j$ and $Q' := Q \cap R_j$, we have $P' \neq Q'$, yet $P' \cap C = P \cap C = \p = Q \cap C = Q' \cap C$.  Thus, $\p$ has a non-singleton fiber under the map $\Spec R_j \ra \Spec C$, contradicting our assumption.  Hence $P = Q$, so only one prime of $T$ contracts to $\p$.
\end{proof}

\begin{cor}
Let $C \subseteq B$ be an integral extension.  Assume either that (1) $C$ has only finitely many minimal primes and $\dim C=0$, or that (2) there is some minimal prime $\p$ of $C$ that satisfies the conditions of Proposition~\ref{pr:contract1} with respect to the extension $C \subseteq B$.  Then there is a \emph{globally} fragile subring $A$ of $B$ containing $C$ such that in the map $\Spec A \ra \Spec C$, every minimal element of $\Spec C$ has a singleton fiber. 
\end{cor}
\begin{proof}
Combine Proposition~\ref{pr:fragile} with Proposition~\ref{pr:dim0} or Proposition~\ref{pr:contract1}.
\end{proof}

Continuing our theme that fragility is closely related to perinormality in pullbacks, we present a result that will later allow us to show that there is no minimal ``perinormalization'' of an integral domain, despite the fact that there is a canonical normalization, seminormalization, and weak normalization.   Additionally this next result is a partial converse to Theorem~\ref{thm:fragilepullback} (obtained by adding extra assumptions).

\begin{prop}[Necessity of fragility when pulling back from dim 0]\label{pr:subfield}
Let $S$ be a generalized Krull domain and $J$ an ideal of height at least two, such that $\dim S/J = 0$ and $J$ has only finitely many minimal primes over it. Let $A \subseteq S/J$ be an integral subring, and let $R$ be the pullback of $A \hookrightarrow S/J \twoheadleftarrow S$.  Then $R$ is perinormal if and only if $A$ is a fragile subring of $S/J$.
\end{prop}

\begin{proof}
The ``if'' direction follows from Theorem~\ref{thm:fragilepullback}.

Conversely, suppose $A$ is not fragile in $B=S/J$.
Then there is some $P \in \Spec R$ with $P \supseteq J$ such that $A_{\bar P} \subseteq B_{\bar P}$ is not apparently fragile, where $\bar P = P/J$.  That is, there is a ring $C$ with $A_{\bar P} \subsetneq C \subseteq B_{\bar P}$, such that the map $\Spec C \ra \Spec A_{\bar P}$ is 1:1 on $\Min A$.  But by the dimension assumption and integrality, it follows that $\Spec C$ and $\Spec A_{\bar P}$ are order-isomorphic.  Moreover, note that $C = D_{A\setminus \bar P}$ for some ring $D$ with $A \subset D \subseteq B$.  

Now, let $T$ be the pullback of the diagram $C \hookrightarrow B_{\bar P} \twoheadleftarrow S_P$.  By Lemma~\ref{lem:basic}(\ref{it:poiso}), $\Spec T$ is order isomorphic to the disjoint union of $\Spec S_P \setminus V_{S_P}(J_P)$ with $\Spec C$.  But since $\Spec C$ is order-isomorphic to $\Spec A_P$, and since $R_P$ is the pullback of $A_{\bar P} \hookrightarrow B_{\bar P} \twoheadleftarrow S_P$ by Lemma~\ref{lem:basic}(\ref{it:loc}), it follows that the map $\Spec T \ra \Spec R_P$ is an order-isomorphism, and hence the extension $R_P \subseteq T$ satisfies going-down.  But $T \neq R_P$, even though $T$ is centered on the maximal ideal of $R_P$, whence $R_P$ (and hence also $R$) is not perinormal.
%
%
\end{proof}

For the next example, we will need the following result regarding when a certain type of diagonal embedding is apparently fragile.

\begin{lemma}\label{lem:fragileproduct}
Let $R \subseteq S$ be an inclusion of integral domains.  Then the diagonal map $d: R \rightarrow R\times S$ given by $d(r)=(r,r)$ induces an apparently fragile inclusion if and only if $S$ is algebraic over $R$.
\end{lemma}

\begin{proof}
First assume $S$ is algebraic over $R$.  Let $T$ be a ring with $d(R) \subsetneq T \subseteq R \times S$.  Let $(r,s) \in T \setminus R$.  By subtracting a diagonal element, we may assume $r=0$ and $s\neq 0$, so that $t=(0,s) \in T$.  If $s\in R$, then $(s,0) = (s,s)-(0,s) \in T$, so that $(s,0) \cdot (0,s)=0$ shows that $T$ has zero-divisors.  Otherwise, by algebraicity, there is some equation of the form \[
\sum_{i=0}^n r_i s^{n-i} = 0,
\]
where $n\geq 1$, $r_i \in R$ for all $0\leq i\leq n$, $r_0 \neq 0$, and $r_n \neq 0$ (with the last condition since $S$ is a domain, so that we may divide out extraneous powers of $s$).  Hence, $0 \neq (0,-r_n) = \sum_{i=0}^{n-1} d(r_i) x^{n-i} \in T$.  But then $d(r_n) + (0,-r_n)  = (r_n,0) \in T$ as well, so that $(r_n,0)\cdot (0,-r_n)=0$.  In either case, $T$ is reduced (as it is a subring of the reduced ring $R \times S$), but is not an integral domain.  Thus, $T$ has at least two minimal primes, completing the proof that $d(R)$ is apparently fragile in $R \times S$.

On the other hand, suppose $S$ is transcendental over $R$.  Then there is some $s\in S$ that is transcendental over $R$.  Let $X$ be an indeterminate, and define a map $\phi: R[X] \rightarrow R \times S$ by setting $\phi(r)=(r,r)$ for all $r\in R$ and $\phi(X) = (0,s)$.  If $g\in \ker \phi$, then $0=g(0,s) = (g(0),g(s))$, so that $g(s)=0$, whence $g$ must be the zero polynomial.  Hence, $\phi$ induces an isomorphism between $R[X]$ and $d(R)[(0,s)]$, which is thus an integral domain that is a subring of $R \times S$ strictly containing $d(R)$.  Therefore, the inclusion $d(R) \subseteq R \times S$ is not apparently fragile.
\end{proof}

\begin{example}\label{ex:noperinormalization}
Unlike the case of normality \cite[Theorem 9.1(ii)]{Mats}, and also unlike the cases of seminormality and weak normality \cite{Vit-survey}, there is in general no minimal ``perinormalization'' of an integral domain.  That is, there exist integral domains that are not perinormal, but are  the intersection of (finitely many) perinormal domains with the same fraction field.

To see this  let $K=\mathbb{Q}(\sqrt[3]{2})$ and  let $F$ be the splitting field of $t^3-2$ (as a subfield of $\mathbb{C}$).
 Define a map $\theta: K\to F$, that sends $\sqrt[3]{2}$ to a different root of $t^3-2$.
  Let $S = \mathbb{Q}[X,Y]$, where $X$ and $Y$ are commuting indeterminates over $\mathbb{Q}$.  Then there exists polynomials  $f(X)$ and $g(X)$ in $S$ such that if $\m=(f, Y)$ and $\n=(g,Y)$, then $S/\m \cong K$ and $S/\n\cong F$.    Let $J = \m\cap \n$.   Then $S$ and $J$ satisfy the assumptions in Theorem~\ref{thm:fragilepullback}.  Moreover, since $\m$ and $\n$ are distinct maximal ideals,  $S/J \cong K\times F$ by the  Chinese Remainder Theorem.  Notice that  $K\times F$ contains two distinct (but isomorphic) subfields, namely $K$ (as viewed via the diagonal map into $K\times F$), and the set $H:=\{(k,\theta(k)): k\in K\}$.  Also note $K\times F$ is integral over both $K$ and $H$.  Moreover, it follows from Lemma~\ref{lem:fragileproduct} that $K$ and $H$ are apparently fragile (hence globally fragile, by Proposition~\ref{pr:dim0}) subrings of $K\times F$.   Hence if $R$ and $T$ are the pullbacks of $K\to  K\times F \leftarrow S$ and   $H\to  K\times F \leftarrow S$ respectively,  both $R$ and $T$ are perinormal by Theorem~\ref{thm:fragilepullback}.   On the other hand, $R\cap T$ is the pullback of $K\cap H\to  K\times F \leftarrow S$.    Since $K\cap H$ is clearly not fragile in $K\times F$, it follows from Proposition~\ref{pr:subfield} that $R\cap T$ is not perinormal.  
 \end{example}

We now present an example of how Theorem~\ref{thm:fragilepullback} may be used, while also showing that the converse to Proposition~\ref{pr:relative} fails.

\begin{example}\label{ex:badpullback}  Let $k$ be a field and let $D:= k[X,Y,Z]$ where $X,Y$ and $Z$ are  commuting indeterminates over $k$.  Let $S:= k[X,Y,Z]_{(X,Y,Z)}$ and let $Q_1$ and $Q_2$ be the ideals (of $S$) $(X,Y)$ and $(Y,Z)$ respectively.  Let $J:=Q_1\cap Q_2 = (XZ, Y)$.   The images of $X,Z$ in $S/J$ will be denoted $x,z$ respectively (the image of $Y$ is 0).    Finally let $A:=k[x+z]$.  We claim that $A$ is a maximal subdomain of $S/J$, hence (globally, by Proposition~\ref{pr:contract1}) fragile in $S/J$, since $S/J$ is reduced.   To see this, let $f$ be an element of $(S/J) \setminus A$, and let $C = A[f]$.  We will show that $C$ has zero-divisors.  After subtracting off the constant term, we may assume  $f = g(x) + h(z)$, where $g$ and $h$ are polynomials over $k$ with zero constant term.  If $g=h$, then $f=g(x)+g(z) = g(x+z) \in A$ already, so we can dispense with this case.  Note that $g(x+z)$ and $h(x+z)$ are elements of $k[x+z]$.  Accordingly, let $u=f(x,z) - g(x+z) = (g(x)+h(z)) - (g(x) + g(z)) = h(z)-g(z)$, and $v=f(x,z)-h(x+z) = (g(x)+h(z)) - (h(x)+h(z)) = g(x) - h(x)$.  Since both $g$ and  $h$ have constant term zero, it follows that $u$ and $v$ are multiples of $z, x$ respectively.  Furthermore  neither $u$ nor $v$ is zero, since $g$, $h$  as abstract polynomials are distinct and in $S/J$, $x$ and $z$ are each transcendental over $k$ (albeit not jointly).  But $uv=0$, and both of them are in $C$, so this ring is not a domain.

We next claim that $S/J$ is integral over $A$.   This can be seen by noting that $x^2-(x+z)x=z^2-(x+z)z=0$.
Hence, by Theorem~\ref{thm:fragilepullback}, if $R$ is the pullback of $A\to S/J \leftarrow S$, then $R$ is perinormal.   In particular, it is perinormal in $S$.  On the other hand, since $S/J$ is integral over $A$, dim $A=1$, and both of the minimal primes of $S/J$ contract to $(0)$ in $A$, it follows that $S/J$ satisfies going-down over $A$.  As $S/J$ is local and is not a localization of $A$ (it is integral over $A$), we see that $A$ is not perinormal in $S/J$, whence the converse to Proposition~\ref{pr:relative} fails.
\end{example}

\section{Hypersurface contraction}\label{sec:hyper}

Next we examine another pullback construction where the resulting ring is perinormal.  First a preparatory lemma.

Let $S$ be an integral domain and $W\subset S$ a multiplicatively closed subset of $S$.   The saturation of $W$ is easily seen to be  the set of all $s\in S$ such that $w/s\in W$ for some $w\in W$.  Moreover, if $W$ is saturated, and if $s/w$ is a unit of $ S_W$, then $s\in W$.   To see this let $s/w, s\in S, w\in W$, be a unit of $S_W$.  Hence there exists  $t/w' \in S_W$ with $w'\in W$, such that $st/ww' =1$ or $st = ww'$.  As $W$ is saturated, we have $s$ (and $t$) in $W$.

\begin{lemma} \label{lem:saturated} Let $S$ be an integral domain such that the set of units of $S$ along with $0$ form a field $k$.    Let $J$ be a  prime ideal of $S$ such that  that all units of $S/J$ are in the image of $k$ (also denoted $k$).  Let $R$ be the pullback of $k\to S/J \leftarrow S$ (or $R:=k\times_{S/J} S$). Then $W=R\setminus J$ is a saturated multiplicatively closed subset of $S$.
\end{lemma}

\begin{proof} First note that $R = k+J$.   Moreover, $W$ consists of those elements of $R$ of the form $v+j$ where $0\neq v \in k$ and $j\in J$.   Let $s\in S$ be in the saturation of $W$ (within $S$).  Then there exists $s'\in S$ such that $ss'\in W$.  Let $\pi$ denote the canonical projection $\pi: S\to S/J$.  Since elements of $W$ map to $k^\times$ in $S/J$, it follows that $\pi(s)$ is a unit of $S/J$.  Hence $\pi(s) \in k$.   Therefore $s-\pi(s) =j \in J$, so that $s=\pi(s)+j \in W$.
\end{proof}

\begin{thm}[Hypersurface contraction]\label{thm:hyper} Let $S$ be a generalized Krull domain, with dim $S > 1$, such that the set of units of $S$ along with $0$ forms a field $k$.    Let $J$ be a principal height one prime ideal of $S$ that is not a maximal ideal of $S$.  Suppose further that all units of $S/J$ are in the image of $k$ (also denoted $k$).  If $R$ is the pullback of $k\to S/J \leftarrow S$ (or $R:=k\times_{S/J} S$), then $R$ is perinormal.
\end{thm}
\begin{proof}      As before we have a bijection between the prime ideals of $S$ that do not contain $J$ and the prime ideals of $R$ that do not contain $J$ (Lemma \ref{lem:basic} (\ref{it:poiso})).   Since $J$ is clearly a maximal ideal of $R$, the bijection is between $\Spec S\setminus V(J)$ and $\Spec R\setminus \{J\}$.  Moreover
 if $Q\in \Spec S\setminus V(J)$ and $P:= Q\cap R$, then  $R_P =  S_Q$ (Lemma \ref{lem:basic} (\ref{it:locoutside})).
 In particular, if $P\in \Spec R$ and $P\neq J$, then $R_P = S_Q$ for some prime ideal $Q$ of $S$, whence $R_P$ is perinormal (since it is a generalized Krull domain).  Therefore  to show that $R$ is perinormal it suffices to show that if $T$ is a local overring of $R$ centered on $J$ that  satisfies going-down over $R$, then $T= R_J$.

 Clearly $R_J \subseteq T$.
 Let $W= R \setminus J$.   Then $JR_W = JS_W$, and by Lemma \ref{lem:basic} (\ref{it:loc}),  $R_J=R_W$ is the pullback of $k\to S_W/JS_W \leftarrow S_W$.  We next claim that ht$_{R_W}JR_W=$ ht$_{R}J >1$.   The equality is clear.  To see the inequality
  let $J\subset M$, where $M$ is a maximal ideal of $S$ (such an $M$ exists, since we are assuming that $J$ is not maximal in $S$).          Now let $P$ be any height one prime ideal of $S$ that is contained in $M$, where $P\neq J$.   Such primes exist since $S_M$ is the intersection of $S$ localized at the height one primes contained in $M$.  Then $0\neq P\cap R \subset J$, which proves the claim.      It is clear that as before we have a bijection between $\Spec(S_W)\setminus V(JS_W)$ and $\Spec(R_W)\setminus \{JR_W\}$.
   Since ht$_{R_W}JR_W \neq 1$, this induces a bijection between $\Spec^1(S_W)\setminus \{JS_W\}$ and $\Spec^1 (R_W)$, with corresponding localizations equal.

    Since $R \subseteq T$ satisfies going-down, each height one prime of $R$ that is contained in $J$ is lain over by a prime ideal of $T$.   Say $Q\cap R=P \in \Spec^1 R$.  Then $R_P \subseteq T_Q$.   Since $R_P$ is a rank 1 valuation ring and $T_Q$ is not a field, we have equality.  Thus if $X=\{Q\in \Spec T: Q\cap R_W \in \Spec^1 R_W\}$ and $Y=\{P\in \Spec^1(R):P\subset J\}$, then
 $$T \subseteq\bigcap_{Q\in X}T_Q  =\bigcap_{P\in Y} R_P.$$

As $J$ is a principal ideal of $S$, we can write $JS_W =fS_W$ for some $f\in S$.   Thus $f$ is contained in only one height one prime of $S_W$ (namely $JS_W$).   The ring $(S_W)_f$ (i.e., the ring $S_W$ localized at the powers of $f$) is a generalized Krull domain, and so it is the intersection of all its localizations at height one prime ideals and each of these is equal to the localization of $R$ at a height one prime contained in $J$.   When combined with the previous containment we have
$$T \subseteq \bigcap_{P\in Y} R_P= (S_W)_f.$$
We will show that the only local ring between $R_W $ and $(S_W)_f$ is $R_W$.

Suppose that this is false and let $g\in T\setminus R_W$.  We note that either $g$ or $g-1$ is a unit of $T$, and neither one is in $R_W$.   Without loss of generality, we can assume that $g$ is a unit.  Since $T\subseteq (S_W)_f$, $g$ is a unit of $(S_W)_f$.  Hence, $g =uf^n$ for some unit $u$ of $S_W$ and $n\in \Z$.  By Lemma~\ref{lem:saturated}, $u=w'/w$ for some $w,w' \in W$.  If $n\geq 0$, then since $W \subseteq R$, $g=(w'/w)f^n\in R_W$, contradicting our assumption on $g$.    Finally if $n=-m<0$ (so that $m>0$), then we have $w'=wgf^m$, with all displayed elements in $T$ and $w'$ a unit, whence $f$ is a unit of $T$, which remains a contradiction.

Thus $g$ cannot exist, whence $T=R_W$.  Therefore $R$ is perinormal.
\end{proof}

\section*{Acknowledgments}
We wish to thank Tiberiu Dumitrescu, who along with several smaller changes found an important error in the previous version of Theorem~\ref{thm:fragilepullback}, leading to a rethinking of the notion of fragility.

\providecommand{\bysame}{\leavevmode\hbox to3em{\hrulefill}\thinspace}
\providecommand{\MR}{\relax\ifhmode\unskip\space\fi MR }
\providecommand{\MRhref}[2]{%
  \href{http://www.ams.org/mathscinet-getitem?mr=#1}{#2}
}
\providecommand{\href}[2]{#2}

\end{document}